\documentclass[12pt, a4paper, oneside]{amsart}
\usepackage{amsmath, amsthm, amssymb, geometry, mathrsfs, hyperref, tikz-cd}
\usepackage[shortlabels, inline]{enumitem}

\newtheorem{theorem}{Theorem}[section]
\newtheorem{proposition}[theorem]{Proposition}

\newtheorem{lemma}[theorem]{Lemma}
\newtheorem{corollary}[theorem]{Corollary}

\theoremstyle{definition}
\newtheorem{definition}[theorem]{Definition}
\newtheorem{question}[theorem]{Question}

\newtheorem{remark}[theorem]{Remark}

\newcommand\C{\mathbb{C}}

\newcommand\N{\mathbb{N}}

\newcommand{\der}{\mathrm{d}}

\newcommand{\id}{\mathrm{id}}
\newcommand{\pr}{\mathrm{pr}}

\newcommand{\T}{\mathrm{T}}

\newcommand{\J}{\mathrm{J}}
\renewcommand{\j}{\mathrm{j}}

\DeclareMathOperator{\Aut}{Aut}

\DeclareMathOperator{\Sing}{Sing}
\DeclareMathOperator{\rank}{rank}
\DeclareMathOperator{\Hom}{Hom}

\newcommand{\CAP}{\mathrm{CAP}}

\newcommand{\aCAP}{\mathrm{aCAP}}

\begin{document}

\title[Thom's jet transversality theorem for regular maps]{Thom's jet transversality theorem \\ for regular maps}
\author{Yuta Kusakabe}
\address{Department of Mathematics, Graduate School of Science, Osaka University, Toyonaka, Osaka 560-0043, Japan}
\email{y-kusakabe@cr.math.sci.osaka-u.ac.jp}
\subjclass[2020]{Primary 14R20, 32Q56, 58A20; Secondary 32E30}
\keywords{transversality theorem, algebraically Oka manifold, flexibility, spray}

\begin{abstract}
We establish Thom's jet transversality theorem for regular maps from an affine algebraic manifold to an algebraic manifold satisfying a suitable flexibility condition.
It can be considered as the algebraic version of Forstneri\v{c}'s jet transversality theorem for holomorphic maps from a Stein manifold to an Oka manifold.
Our jet transversality theorem implies genericity theorems for regular maps of maximal ranks.
As an application, it follows that every connected compact locally flexible manifold is the image of a holomorphic submersion from an affine space.
We also show that every algebraically degenerate subvariety of codimension at least two in a locally flexible manifold has an Oka complement.
\end{abstract}

\maketitle

%
%

\section{Introduction}

Thom's jet transversality theorem \cite{Thom1956} is one of the most fundamental tools in differential geometry.
In 2006, Forstneri\v{c} \cite{Forstneric2006a} proved the complex analytic version of his theorem for holomorphic maps from a Stein manifold to an Oka manifold (see also \cite{Kaliman1996}).
Here, a complex manifold $Y$ is an \emph{Oka manifold} if for any holomorphic map $f:X\to Y$ from a Stein manifold there exists a holomorphic map $s:X\times\C^{N}\to Y$ such that $s(x,0)=f(x)$ and $s(x,\cdot):\C^{N}\to Y$ is a submersion at $0$ for each $x\in X$ (cf. \cite{Forstneric2017,Kusakabe2019,Kusakabe}).

In the present paper, we establish the complex algebraic version of Thom's jet transversality theorem.
To this end, let us recall the following definition.

\begin{definition}[{cf. \cite{Larusson2019}}]
\label{definition:aEll_1}
An algebraic manifold $Y$ is an \emph{algebraically Oka manifold} if for any regular map $f:X\to Y$ from an affine manifold there exists a regular map $s:X\times\C^{N}\to Y$ such that $s(x,0)=f(x)$ and $s(x,\cdot):\C^{N}\to Y$ is a submersion at $0$ for each $x\in X$.
\end{definition}

The most typical examples of algebraically Oka manifolds are locally flexible manifolds (cf. \cite{Arzhantsev2013,Kaliman2018}).
A quasi-affine manifold $Y$ is \emph{flexible} if for any $y\in Y$ the tangent space $\T_{y}Y$ is spanned by the tangent vectors to the orbits of one-parameter unipotent subgroups of $\Aut Y$ through $y$.
An algebraic manifold is \emph{locally flexible} if it is covered by flexible Zariski open subsets.
Complex Grassmannians and smooth nondegenerate toric varieties (i.e. smooth toric varieties with no torus factor) are known to be locally flexible \cite[Theorem 3]{Larusson2019} (see \cite[\S3]{Arzhantsev2013a} for more examples of flexible manifolds).
Currently, it is not known whether there exists an algebraically Oka manifold which is not locally flexible.

For regular maps from an affine manifold to an algebraically Oka manifold, Forstneri\v{c} proved the \emph{local} jet transversality theorem \cite[Theorem 4.3]{Forstneric2006a}.
Under an ostensibly stronger assumption on the target (that is weaker than local flexibility), we prove the following \emph{global} jet transversality theorem which generalizes \cite[Proposition 4.10]{Forstneric2006a} (compare with \cite[Theorem 4.9]{Forstneric2006a} and \cite[Corollary 1.3]{Kaliman1996}).
Here, $\J^{k}(X,Y)$ denotes the space of $k$-jets of holomorphic maps $X\to Y$, $\j^{k}f:X\to\J^{k}(X,Y)$ denotes the $k$-jet extension of a regular map $f:X
\to Y$ and $\Sing(s)$ denotes the singular locus of a regular map $s$ (i.e. the set of points where $s$ is not a submersion).

\begin{theorem}
\label{theorem:main}
Let $Y$ be an algebraic manifold.
Assume that for any $y\in Y$ there exist a Zariski open neighborhood $U\subset Y$ of $y$ and a regular map $s:U\times\C^{N}\to Y$ such that $\dim\Sing(s)<N$, $s(y,0)=y$ and $s(y,\cdot):\C^{N}\to Y$ is a submersion at $0$ for each $y\in U$.
Then regular maps from affine manifolds to $Y$ satisfy the following jet transversality theorem:

Let $X$ be an affine manifold, $X_{0}\subset X$ be a closed algebraic subvariety, $k_{l}$ $(l\in\N)$ be nonnegative integers, $A_{l}\subset X$ and $B_{l}\subset\J^{k_{l}}(X,Y)$ $(l\in\N)$ be (not necessarily closed or algebraic) complex submanifolds, and $f_{0}:X\to Y$ be a regular map such that $\j^{k_{l}}f_{0}|_{A_{l}}$ is transverse to $B_{l}$ at all points of $X_{0}\cap A_{l}$ for each $l\in\N$.
Then for any integer $k\geq\sup_{l\in\N}k_{l}$ the regular map $f_{0}$ can be approximated  uniformly on compacts by regular maps $f:X\to Y$ such that
\begin{enumerate}
\item $\j^{k}f|_{X_{0}}=\j^{k}f_{0}|_{X_{0}}$, and
\item $\j^{k_{l}}f|_{A_{l}}$ is transverse to $B_{l}$ for each $l\in\N$.
\end{enumerate}
\end{theorem}

Since $A_{l}$ and $B_{l}$ $(l\in\N)$ are not assumed to be closed, we can also take stratified complex subvarieties $A_{l}$ and $B_{l}$ $(l\in\N)$ in Theorem \ref{theorem:main} (see \cite[\S4]{Forstneric2006a} for the definition of transversality in this setting).
The assumption in Theorem \ref{theorem:main} implies that $Y$ is algebraically Oka (by Lemma \ref{lemma:extension} and \cite[Theorem 1]{Larusson2019}).
Conversely, if the conclusion of Theorem \ref{theorem:main} holds for an algebraically Oka manifold $Y$, it satisfies the assumption in Theorem \ref{theorem:main} (Remark \ref{remark:assumption}).
Thus the assumption in Theorem \ref{theorem:main} is natural.
Since every locally flexible manifold satisfies this assumption (Proposition \ref{proposition:flexible}), the following corollary holds.

\begin{corollary}
Regular maps from any affine manifold to any locally flexible manifold satisfy the jet transversality theorem.
\end{corollary}

The proof of Theorem \ref{theorem:main} is given in Section \ref{section:proof}.
As immediate consequences of Theorem \ref{theorem:main}, we can obtain the following genericity results by the same proof as in \cite[Corollary 8.9.3]{Forstneric2017} (see the proof of \cite[Theorem 1.6]{Forstneric2017a} for (3)).
Here, $\Hom(X,Y)$ denotes the space of regular maps $X\to Y$ equipped with the compact-open topology.

\begin{corollary}
\label{corollary:genericity}
Assume that $X$ is an $n$-dimensional affine manifold and $Y$ is an $m$-dimensional algebraic manifold satisfying the assumption in Theorem \ref{theorem:main}.
Then the following hold:
\begin{enumerate}[leftmargin=*]
\item If $2n\leq m$, then the set of immersions $X\to Y$ is dense in $\Hom(X,Y)$.
\item If $2n+1\leq m$, then the set of injective immersions $X\to Y$ is dense in $\Hom(X,Y)$.
\item If $n\geq m$, then for any compact subset $K\subset Y$ the set of regular maps $f:X\to Y$ satisfying $K\subset f(X)$ and $\dim\Sing(f)<m$ is dense in $\Hom(X,Y)$.
\end{enumerate}
\end{corollary}

By using Theorem \ref{theorem:main}, we can also obtain the following approximation theorem for holomorphic maps with a lower bound on the rank which generalizes \cite[Theorem 9.12.4]{Forstneric2017} (see the first part of the proof of \cite[Theorem 9.12.4]{Forstneric2017}; see also \cite{Kolaric2008}).
Note that it contains approximation theorems for immersions ($n=r<m$) and submersions ($n>m=r$) as special cases (see \cite[Problem 9.14.3]{Forstneric2017}).

\begin{corollary}
\label{corollary:approximation}
Assume that $n,m,r$ are integers such that $(n-r+1)(m-r+1)\geq 2$ and $Y$ is an $m$-dimensional algebraic manifold satisfying the assumption in Theorem \ref{theorem:main}.
Then any holomorphic map $f:U\to Y$ from an open neighborhood of a compact convex set $K\subset\C^{n}$ satisfying $\inf_{z\in K}\rank\der f_{z}\geq r$ can be approximated uniformly on $K$ by holomorphic maps $\tilde f:\C^{n}\to Y$ such that $\inf_{z\in\C^{n}}\rank\der \tilde f_{z}\geq r$.
\end{corollary}

Forstneri\v{c} proved that every $n$-dimensional connected compact algebraically Oka manifold is the image of a strongly dominating regular map from $\C^{n}$ \cite[Theorem 1.6]{Forstneric2017a}.
Recall that every $n$-dimensional connected complex manifold is the image of a locally biholomorphic map from the unit polydisc in $\C^{n}$ by the result of Forn\ae ss and Stout \cite{Fornaess1977} (see also \cite{Fornaess1982}).
This result and Corollary \ref{corollary:approximation} imply the following by Rouch\'{e}'s theorem (cf. \cite[p.\,110]{Chirka1989}).

\begin{corollary}
Assume that $Y$ is an $n$-dimensional connected algebraic manifold satisfying the assumption in Theorem \ref{theorem:main}.
Then for any compact subset $K\subset Y$ there exists a holomorphic submersion $f:\C^{n+1}\to Y$ such that $K\subset f(\C^{n+1})$.
In particular, every $n$-dimensional connected compact locally flexible manifold is the image of a holomorphic submersion from $\C^{n+1}$.
\end{corollary}

In Section \ref{section:applications}, we give another application of Theorem \ref{theorem:main}.
In Oka theory, it is important to understand when a closed complex subvariety in an Oka manifold has an Oka complement.
By the Kobayashi conjecture \cite{Kobayashi2005}, it is natural to assume that the codimension of a subvariety is at least two.
It is known that there exists a discrete set $D\subset\C^{2}$ such that $\C^{2}\setminus D$ is not Oka \cite[Theorem 4.5]{Rosay1988}, and that every algebraically degenerate\footnote{A closed complex subvariety in an algebraic manifold $Y$ is \emph{algebraically degenerate} if it is contained in a proper closed algebraic subvariety of $Y$.} closed complex subvariety $A\subset\C^{n}$ of codimension at least two has an Oka complement \cite[Theorem 1.6]{Forstneric2002}.
We obtain the following generalization of the latter.

\begin{corollary}
\label{corollary:complement}
Let $Y$ be an algebraic manifold satisfying the assumption in Theorem \ref{theorem:main} and $A\subset Y$ be a closed complex subvariety of codimension at least two.
Assume that there exists a holomorphic automorphism $\varphi$ of $Y$ such that $\varphi(A)$ is algebraically degenerate.
Then the complement $Y\setminus A$ is Oka.
\end{corollary}

Since a connected affine algebraic group without nontrivial characters is flexible \cite[Proposition 5.4]{Arzhantsev2013}, Corollary \ref{corollary:complement} generalizes the result of Winkelmann \cite[Theorem 2.9]{Winkelmanna} for the complement of a tame discrete set in such an algebraic group (see \cite[Theorem 2.1]{Winkelmanna}).

It is known that local flexibility is preserved by removing closed algebraic subvarieties of codimension at least two \cite[Theorem 1.1]{Flenner2016}.
Thus it is reasonable to expect that the algebraic Oka property is also preserved by the same operation (cf. \cite[Problem 6.4.3]{Forstneric2017}).
At present, however, we can only conclude that a closed algebraic subvariety of codimension at least two in an algebraic manifold satisfying the assumption in Theorem \ref{theorem:main} has a complement enjoying aCAP (a Runge type approximation property which is implied by the algebraic Oka property; see Definition \ref{definition:CAP} and Remark \ref{remark:aCAP}).

%
%

\section{Singularities of algebraic sprays}
\label{section:sing}

In this section, we recall the notion of algebraic sprays and investigate their singularities to prove Theorem \ref{theorem:main}.
For a regular map $p:E\to X$, let $E_{x}$ denote the fiber $p^{-1}(x)$ over a point $x\in X$.

\begin{definition}
\label{definition:spray}
Let $X$ and $Y$ be algebraic manifolds.
\begin{enumerate}[leftmargin=*]
\item A \emph{(global) algebraic spray} over a regular map $f:X\to Y$ is a triple $(E,p,s)$ where $p:E\to X$ is an algebraic vector bundle and $s:E\to Y$ is a regular map such that $s(0_{x})=f(x)$ for all $x\in X$.
If the algebraic vector bundle $p:E\to X$ is trivial, we simply write $s:E\to Y$ instead of $(E,p,s)$.
\item A family of algebraic sprays $\{(E_{j},p_{j},s_{j})\}_{j=1}^{k}$ over $f:X\to Y$ is \emph{dominating} if $\sum_{j=1}^{k}(\der s_{j})_{0_{x}}(E_{j,x})=\T_{f(x)}Y$ holds
for each $x\in X$.
\end{enumerate}
\end{definition}

Note that the regular map $s:X\times\C^{N}\to Y$ in Definition \ref{definition:aEll_1} (resp. $s:U\times\C^{N}\to Y$ in Theorem \ref{theorem:main}) is nothing but a dominating algebraic spray over $f:X\to Y$ (resp. the inclusion $U\hookrightarrow Y$).

Let us first recall the following lemma which was used in the proof of Gromov's localization principle for algebraically Oka manifolds \cite[3.5.B]{Gromov1989} (see also \cite[Proposition 6.4.2]{Forstneric2017} and \cite[Proposition 1.4]{Kaliman2018}).

\begin{lemma}
\label{lemma:extension}
Assume that $Y$ is an algebraic manifold, $D\subset Y$ is an algebraic hypersurface and $(E,p,s)$ is an algebraic spray over the inclusion $Y\setminus D\hookrightarrow Y$.
Then there exist an algebraic spray $(\widetilde E,\tilde p,\tilde s)$ over the identity map $\id_{Y}$ and an isomorphism $\varphi:E\to\widetilde E|_{Y\setminus D}$ of algebraic vector bundles such that $\tilde s\circ\varphi=s$ and $\Sing(\tilde s)\cap\widetilde E|_{D}=\emptyset$.
\end{lemma}

\begin{proof}
See the proof of \cite[Proposition 1.4]{Kaliman2018} for the construction of an algebraic spray $(\widetilde E,\tilde p,\tilde s)$ over the identity map $\id_{Y}$ and an isomorphism $\varphi:E\to\widetilde E|_{Y\setminus D}$ of algebraic vector bundles such that $\tilde s\circ\varphi=s$.
By construction, we may assume that for any $e\in\widetilde E|_{D}$ there exists a local holomorphic section $f:U\to\widetilde E$ from an open neighborhood $U\subset Y$ of $\tilde p(e)$ such that $f(\tilde p(e))=e$ and $\der(\tilde s\circ f)_{\tilde p(e)}=\id_{\T_{\tilde p(e)}Y}$.
This implies $\Sing(\tilde s)\cap\widetilde E|_{D}=\emptyset$.
\end{proof}

By using the assumption in Theorem \ref{theorem:main} and the above lemma, we can construct a dominating family of algebraic sprays over the identity map with small singular loci.

\begin{corollary}
\label{corollary:dimSing(s)<rankE}
Assume that $Y$ is an algebraic manifold satisfying the assumption in Theorem \ref{theorem:main}.
Then for any $y\in Y$ there exists an algebraic spray $(E,p,s)$ over the identity map $\id_{Y}$ such that $\dim\Sing(s)<\rank E$ and $0_{y}\not\in\Sing(p,s)$.
\end{corollary}

\begin{proof}
Take a point $y\in Y$ arbitrarily.
By assumption, there exist a Zariski open neighborhood $U\subset Y$ of $y$ and a dominating algebraic spray $s_{0}:U\times\C^{N}\to Y$ over the inclusion $U\hookrightarrow Y$ satisfying $\dim\Sing(s_{0})<N$.
After shrinking $U\ni y$ if necessary, we may assume that $D=Y\setminus U$ is an algebraic hypersurface.
Then Lemma \ref{lemma:extension} implies that the algebraic spray $s_{0}:U\times\C^{N}\to Y$ extends to an algebraic spray $(E,p,s)$ over the identity map $\id_{Y}$ such that $\Sing(s)\cap E|_{D}=\emptyset$.
This extension clearly satisfies $\dim\Sing(s)<N=\rank E$ and $0_{y}\not\in\Sing(p,s)$.
\end{proof}

In order to prove Theorem \ref{theorem:main}, we need to recall Gromov's method of composed sprays \cite[\S1.3]{Gromov1989} (see also \cite[\S6.3]{Forstneric2017}).

\begin{definition}[{cf. \cite[Definition 6.3.5]{Forstneric2017}}]
Let $X,Y$ and $Z$ be algebraic manifolds.
\begin{enumerate}[leftmargin=*]
\item For a family of algebraic sprays $\{(E_{j},p_{j},s_{j})\}_{j=1}^{k}$ over the identity map $\id_{Y}$, the \emph{composed spray} $(E_{1}*\cdots*E_{k},p_{1}*\cdots*p_{k},s_{1}*\cdots*s_{k})$ is defined by
\begin{align*}
E_{1}*\cdots*E_{k}&=\left\{(e_{1},\ldots,e_{k})\in\prod_{j=1}^{k}E_{j}:s_{j}(e_{j})=p_{j+1}(e_{j+1}),\ j=1,\ldots,k-1\right\},\\
(p_{1}*\cdots&*p_{k})(e_{1},\ldots,e_{k})=p_{1}(e_{1}),\quad (s_{1}*\cdots*s_{k})(e_{1},\ldots,e_{k})=s_{k}(e_{k}).
\end{align*}
\item For an algebraic spray $(E,p,s)$ over the identity map $\id_{Y}$ and a natural number $k\in\N$, the $k$-th \emph{iterated spray} $(E^{(k)},p^{(k)},s^{(k)})$ is the composed spray $(E*\cdots*E,p*\cdots*p,s*\cdots*s)$ of $k$ copies of $(E,p,s)$.
\item For a regular map $f:X\to Y$ and an algebraic spray $(E,p,s)$ over a regular map $g:Y\to Z$, the \emph{pullback spray} $(f^{*}E,f^{*}p,f^{*}s)$ over $g\circ f:X\to Z$ consists of the pullback bundle $f^{*}p:f^{*}E\to X$ and the composition $f^{*}s:f^{*}E\to Z$ of the induced map $f^{*}E\to E$ and $s:E\to Z$.
\end{enumerate}
\end{definition}

\begin{remark}
Note that the projection $p_{1}*\cdots*p_{k}:E_{1}*\cdots*E_{k}\to Y$ of a composed spray $(E_{1}*\cdots*E_{k},p_{1}*\cdots*p_{k},s_{1}*\cdots*s_{k})$ does not have any canonical algebraic vector bundle structure.
However, it has the zero section $\{(0_{y},\cdots,0_{y})\in E_{1}*\cdots*E_{k}:y\in Y\}$.
The rank of $E_{1}*\cdots*E_{k}$ is defined to be the dimension of the fibers of $p_{1}*\cdots*p_{k}$ and denoted by $\rank(E_{1}*\cdots*E_{k})$.
\end{remark}

Let us prove two lemmas on the singular locus of a composed spray.

\begin{lemma}
\label{lemma:Sing(s)}
Assume that $Y$ is an algebraic manifold and $\{(E_{j},p_{j},s_{j})\}_{j=1}^{k}$ is a family of algebraic sprays over the identity map $\id_{Y}$ such that $\dim\Sing(s_{j})<\rank E_{j}$ for each $j$.
Set $\Sigma=(E_{1}*\cdots*E_{k})\setminus\prod_{j=1}^{k}(E_{j}\setminus\Sing(s_{j}))$.
Then the following holds:
\begin{align*}
\dim\Sing(s_{1}*\cdots*s_{k})\leq\dim\Sigma<\rank(E_{1}*\cdots*E_{k}).
\end{align*}
\end{lemma}

\begin{proof}
By definition, the singular locus of the composed spray satisfies $\Sing(s_{1}*\cdots*s_{k})\subset\Sigma$ and thus $\dim\Sing(s_{1}*\cdots*s_{k})\leq\dim\Sigma$.

Set $\Sigma_{l}=(E_{l}*\cdots*E_{k})\setminus\prod_{j=l}^{k}(E_{j}\setminus\Sing(s_{j}))$ for $l=1,\ldots,k$.
Note that $\Sigma_{1}=\Sigma$ and $\Sigma_{k}=\Sing(s_{k})$.
For each $l=1,\ldots,k-1$, the following holds:
\begin{align*}
\Sigma_{l}=(E_{l}*\cdots*E_{k})\cap((\Sing(s_{l})\times(E_{l+1}*\cdots*E_{k}))\cup((E_{l}\setminus\Sing(s_{l}))\times\Sigma_{l+1})).
\end{align*}
By definition and assumption,
\begin{align*}
\dim((E_{l}*\cdots*E_{k})\cap(\Sing(s_{l})\times(E_{l+1}*\cdots*E_{k})))&=\dim\Sing(s_{l})+\sum_{j=l+1}^{k}\rank E_{j} \\
&<\sum_{j=l}^{k}\rank E_{j} \\
&=\rank(E_{l}*\cdots*E_{k}).
\end{align*}
Since the restriction $s_{l}:E_{l}\setminus\Sing(s_{l})\to Y$ is a submersion,
\begin{align*}
&\dim((E_{l}*\cdots*E_{k})\cap((E_{l}\setminus\Sing(s_{l}))\times\Sigma_{l+1})) \\
\leq&\dim(E_{l}\setminus\Sing(s_{l}))-\dim Y+\dim\Sigma_{l+1} \\
=&\rank E_{l}+\dim\Sigma_{l+1}.
\end{align*}
Thus the desired inequality follows by downward induction on $l$.
\end{proof}

\begin{lemma}
\label{lemma:Sing(p,s)}
Assume that $Y$ is an algebraic manifold and $\{(E_{j},p_{j},s_{j})\}_{j=1}^{k}$ is a family of algebraic sprays over the identity map $\id_{Y}$.
Set $\Sigma=(E_{1}*\cdots*E_{k})\setminus\prod_{j=1}^{k}(E_{j}\setminus\Sing(s_{j}))$.
Then the following holds:
\begin{align*}
\Sing(p_{1}*\cdots*p_{k},s_{1}*\cdots*s_{k})\setminus\Sigma\subset\prod_{j=1}^{k}\Sing(p_{j},s_{j}).
\end{align*}
\end{lemma}

\begin{proof}
Take a point $(e_{1},\ldots,e_{k})\in\Sing(p_{1}*\cdots*p_{k},s_{1}*\cdots*s_{k})\setminus\Sigma$ arbitrarily.
To reach a contradiction, assume that $(e_{1},\ldots,e_{k})\not\in\prod_{j=1}^{k}\Sing(p_{j},s_{j})$.
Then $e_{l}\not\in\Sing(p_{l},s_{l})$ for some $l$.
This means that $s_{l}|_{E_{l,p_{l}(e_{l})}}:E_{l,p_{l}(e_{l})}\to Y$ is a submersion at $e_{l}$.
Therefore there exists a holomorphic map $f_{l}:U_{l}\to E_{l,p_{l}(e_{l})}$ from an open neighborhood $U_{l}\subset Y$ of $s_{l}(e_{l})$ such that $s_{l}\circ f_{l}=\id_{U_{l}}$ and $f_{l}(s_{l}(e_{l}))=e_{l}$.
Since $e_{j}\not\in\Sing(s_{j})$ $(j=l+1,\ldots,k)$, there exist holomorphic maps $f_{j}:U_{j}\to E_{j}$ $(j=l+1,\ldots,k)$ from open neighborhoods of $U_{j}\subset Y$ of $s_{j}(e_{j})$ such that $s_{j}\circ f_{j}=\id_{U_{j}}$ and $f_{j}(s_{j}(e_{j}))=e_{j}$.
Note that $s_{j}(e_{j})=p_{j+1}(e_{j+1})$ holds for each $j$.
Thus after shrinking $U_{k}\ni s_{k}(p_{k})$ if necessary, we can consider the holomorphic map $f:U_{k}\to(E_{1}*\cdots*E_{k})_{(p_{1}*\cdots*p_{k})(e_{1},\ldots,e_{k})}$ defined by
\begin{align*}
f(y)=(e_{1},\ldots,e_{l-1},f_{l}\circ(p_{l+1}\circ f_{l+1})\circ\cdots\circ(p_{k}\circ f_{k})(y),\ldots,f_{k}(y)).
\end{align*}
By definition, it satisfies $(s_{1}*\cdots*s_{k})\circ f=\id_{U_{k}}$ and $f((s_{1}*\cdots*s_{k})(e_{1},\ldots,e_{k}))=(e_{1},\ldots,e_{k})$.
This implies $(e_{1},\ldots,e_{k})\not\in\Sing(p_{1}*\cdots*p_{k},s_{1}*\cdots*s_{k})$, a contradiction to our assumption.
\end{proof}

With the above lemmas and Corollary \ref{corollary:dimSing(s)<rankE}, we can construct a dominating composed spray with small relative singular locus under the assumption in Theorem \ref{theorem:main}.

\begin{corollary}
\label{corollary:dimSing(p,s)<rankE}
Assume that $Y$ is an algebraic manifold satisfying the assumption in Theorem \ref{theorem:main}.
Then there exists a dominating composed spray $(E,p,s)$ over the identity $\id_{Y}$ which satisfies $\dim\Sing(p,s)<\rank E$.
\end{corollary}

\begin{proof}
By Corollary \ref{corollary:dimSing(s)<rankE}, there exists a dominating family of algebraic sprays $\{(E_{j},p_{j},s_{j})\}_{j=1}^{k}$ over the identity map $\id_{Y}$ such that $\dim\Sing(s_{j})<\rank E_{j}$ for each $j$.
Then the composed spray $(\widetilde E,\tilde p,\tilde s)$ of the family $\{(E_{j},p_{j},s_{j})\}_{j=1}^{k}$ is dominating and satisfies $\dim\Sing(\tilde s)<\rank\widetilde E$ by Lemma \ref{lemma:Sing(s)}.
Note that the dimensions of the fibers of $\Sing(\tilde p,\tilde s)$ satisfy $\sup_{y\in Y}\dim\Sing(\tilde p,\tilde s)_{y}\leq\rank\widetilde E-1$ since the spray $(\widetilde E,\tilde p,\tilde s)$ is dominating.
Let us consider the $n$-th iterated spray $(\widetilde E^{(n)},\tilde p^{(n)},\tilde s^{(n)})$ of $(\widetilde E,\tilde p,\tilde s)$ for a natural number $n>\dim Y$.
Set $\Sing(\tilde p,\tilde s)^{(n)}=\widetilde E^{(n)}\cap(\Sing(\tilde p,\tilde s))^{n}$ and $\Sigma=\widetilde E^{(n)}\setminus(\widetilde E\setminus\Sing(\tilde s))^{n}$.
Then the dimensions of the fibers of $\Sing(\tilde p,\tilde s)^{(n)}$ satisfies
\begin{align*}
\sup_{y\in Y}\dim(\Sing(\tilde p,\tilde s)^{(n)})_{y}&\leq n\sup_{y\in Y}\dim\Sing(\tilde p,\tilde s)_{y} \\
&\leq n\rank\widetilde E-n \\
&<\rank \widetilde E^{(n)}-\dim Y,
\end{align*}
and hence $\dim\Sing(\tilde p,\tilde s)^{(n)}<\rank\widetilde E^{(n)}$.
By Lemma \ref{lemma:Sing(s)}, the closed subvariety $\Sigma$ also satisfies $\dim\Sigma<\rank\widetilde E^{(n)}$.
Since $\Sing(\tilde p^{(n)},\tilde s^{(n)})\subset\Sing(\tilde p,\tilde s)^{(n)}\cup\Sigma$ by Lemma \ref{lemma:Sing(p,s)}, it follows that $\dim\Sing(\tilde p^{(n)},\tilde s^{(n)})<\rank\widetilde E^{(n)}$.
Thus $(E,p,s)=(\widetilde E^{(n)},\tilde p^{(n)},\tilde s^{(n)})$ is a desired dominating composed spray.
\end{proof}

%
%

\section{Proof of Theorem \ref{theorem:main}}
\label{section:proof}

Our proof of Theorem \ref{theorem:main} is based on the ideas in \cite[\S4]{Forstneric2006a}.
Let $\pr_{X_{\lambda_{0}}}:\prod_{\lambda\in\Lambda}X_{\lambda}\to X_{\lambda_{0}}$ denote the projection map to $X_{\lambda_{0}}\ (\lambda_{0}\in\Lambda)$.

\begin{proof}[Proof of Theorem \ref{theorem:main}]
By Corollary \ref{corollary:dimSing(p,s)<rankE}, there exists a dominating composed spray $(E,p,s)$ over the identity $\id_{Y}$ satisfying $\dim\Sing(p,s)<\rank E$.
Let us consider the pullback spray $(f_{0}^{*}E,f_{0}^{*}p,f_{0}^{*}s)$.
By the argument in the proof of \cite[Lemma 3.6]{Forstneric2006a}, there exists an algebraic submersion (i.e. a smooth morphism) $\varphi:X\times\C^{N}\to f_{0}^{*}E$ over $X$ which preserves the zero section.
Take regular functions $g_{1},\ldots,g_{L}:X\to\C$ which vanish to order $k+1$ on the subvariety $X_{0}=\{x\in X:g_{j}(x)=0,\ j=1,\ldots,L\}$.
Consider the regular map $s_{X_{0}}:X\times(\C^{N})^{L}\to X\times\C^{N}$, $(x,t_{1},\ldots,t_{L})\mapsto (x,g_{1}(x)t_{1}+\cdots+g_{L}(x)t_{L})$ over $X$.
Note that $s_{X_{0}}(x,\cdot):(\C^{N})^{L}\to\{x\}\times\C^{N}\cong\C^{N}$ is a submersion if $x\in X\setminus X_{0}$, and it vanishes if $x\in X_{0}$.
Since $X$ is affine, we may assume that $X$ is a closed algebraic submanifold of $\C^{n}$ for some $n\in\N$.
Let $W$ denote the complex vector space of polynomial maps $\C^{n}\to\C^{N+L}$ of degree at most $k$, and $s_{W}:X\times W\to X\times\C^{N+L}$ be the algebraic submersion over $X$ defined by $s_{W}(x,P)=(x,P(x))$.
We define the algebraic spray $s_{f_{0}}:X\times\C^{N+L}\to Y$ over $f_{0}$ by $s_{f_{0}}=f_{0}^{*}s\circ\varphi\circ s_{X_{0}}\circ s_{W}$ (we are identifying the zero sections with $X$ in the following diagram):
\begin{center}
\begin{tikzcd}
X \arrow[d, hook] \arrow[rr, equal] &  & X \arrow[d, hook] \arrow[drr, bend left=15, "f_{0}"]  \\
X\times W \arrow[d] \arrow[rr, "{\varphi\circ s_{X_{0}}\circ s_{W}}"] & & f_{0}^{*}E \arrow[r, "p^{*}f_{0}"] \arrow[d] & E \arrow[d] & Y \arrow[from=llll, dashed, crossing over, bend right=18, "s_{f_{0}}"' near start] \arrow[from=l, "s"]  \\
X \arrow[rr, equal] & & X \arrow[r, "f_{0}"] & Y \arrow[ul, phantom, "\ulcorner", very near start]
\end{tikzcd}
\end{center}
Note that $\Sing(\pr_{X},s_{f_{0}})\setminus(X_{0}\times W)=(p^{*}f_{0}\circ\varphi\circ s_{X_{0}}\circ s_{W})^{-1}(\Sing(p,s))\setminus(X_{0}\times W)$ holds by construction.
From this and $\dim\Sing(p,s)<\rank E$, it follows that there exists a dense Zariski open subset $U\subset W$ such that $\Sing(\pr_{X},s_{f_{0}})\cap(X\times U)\subset X_{0}\times U$.

For each $l\in\N$, let us consider the map $\Phi_{k_{l}}:X\times W\to\J^{k_{l}}(X,Y)$ defined by $\Phi_{k_{l}}(x,P)=\j^{k_{l}}(s_{f_{0}}\circ(\id_{X},P))(x)$.
Then by the proof of \cite[Lemma 4.5]{Forstneric2006a}, the singular locus of $(\pr_{X},\Phi_{k_{l}})$ satisfies $\Sing(\pr_{X},\Phi_{k_{l}})\subset\Sing(\pr_{X},s_{f_{0}})$, and thus $\Sing(\pr_{X},\Phi_{k_{l}})\cap(X\times U)\subset X_{0}\times U$.
Let us consider the complex submanifolds $A_{l}\subset X$ and $B_{l}\subset\J^{k_{l}}(X,Y)$.
For $P\in U$, we can easily see that the regular map $\Phi_{k_{l}}(\cdot,P)|_{A_{l}\setminus X_{0}}:A_{l}\setminus X_{0}\to\J^{k_{l}}(X,Y)$ is transverse to $B_{l}$ if and only if $P\not\in\pr_{U}(\Sing(\pr_{U}|_{B_{l}'}))$ where $B_{l}'=(\Phi_{k_{l}}|_{(A_{l}\setminus X_{0})\times U})^{-1}(B_{l})$.
By Sard's theorem, there exists a polynomial map $P\in U\setminus\bigcup_{l\in\N}\pr_{U}(\Sing(\pr_{U}|_{B_{l}'}))$ which is sufficiently close to $0$.
Then by construction, the regular map $f=s_{f_{0}}(\cdot,P):X\to Y$ satisfies $\j^{k}f|_{X_{0}}=\j^{k}f_{0}|_{X_{0}}$, and thus $\j^{k_{l}}f|_{A_{l}}$ is transverse to $B_{l}$ for each $l\in\N$.
Furthermore, it also approximates $f_{0}$ since $P$ is close to 0.
\end{proof}

In the rest of this section, we give a few remarks on the assumption in Theorem \ref{theorem:main}.

\begin{remark}
\label{remark:assumption}
As we mentioned in the introduction, if the jet transversality theorem holds for regular maps from affine manifolds to an algebraically Oka manifold $Y$, then $Y$ satisfies the assumption in Theorem \ref{theorem:main}.
Indeed, for any $y\in Y$ and any affine Zariski open neighborhood $U\subset Y$ of $y$ there exists a dominating algebraic spray $s_{0}:U\times\C^{N}\to Y$ over the inclusion $U\hookrightarrow Y$ since $Y$ is algebraically Oka.
By the jet transversality theorem and the proof of (3) in Corollary \ref{corollary:genericity} (see \cite[\S8.9]{Forstneric2017}), there exists a regular map $s:U\times\C^{N}\to Y$ such that $\dim\Sing(s)<\dim Y\leq N$ and $\j^{1}s|_{U\times\{0\}}=\j^{1}s_{0}|_{U\times\{0\}}$, and thus it is a dominating algebraic spray over the inclusion $U\hookrightarrow Y$ with the desired property.
\end{remark}

The above remark leads us to the following question.

\begin{question}
Does every algebraically Oka manifold satisfy the assumption in Theorem \ref{theorem:main}?
\end{question}

Recall that a flexible manifold $Y$ admits a dominating algebraic spray $s:Y\times\C^{N}\to Y$ over the identity map $\id_{Y}$ such that $s(\cdot,w)\in\Aut Y$ for all $w\in\C^{N}$ (see the proof of \cite[Proposition 5.6.22]{Forstneric2017}).
Thus the following proposition holds.

\begin{proposition}
\label{proposition:flexible}
Every locally flexible manifold satisfies the assumption in Theorem \ref{theorem:main}.
\end{proposition}

%
%

\section{Proof of Corollary \ref{corollary:complement}}
\label{section:applications}

Let us first recall the Convex Approximation Property (CAP) and its algebraic
version aCAP.

\begin{definition}
\label{definition:CAP}
A complex manifold (resp. an algebraic manifold) $Y$ enjoys \emph{CAP} (resp. \emph{aCAP}) if any holomorphic map from an open neighborhood of a compact convex set $K\subset\C^{n}$ $(n\in\N)$ to $Y$ can be uniformly approximated on $K$ by holomorphic maps (resp. regular maps) $\C^{n}\to Y$.
\end{definition}

In order to prove Corollary \ref{corollary:complement}, we also need to recall the following fact.

\begin{theorem}[{cf. \cite[Theorem 1.3]{Kusakabe2019}, \cite[Corollary 1.2 and Proposition 4.6]{Forstneric2006a}}]
\label{theorem:CAP}
\ \\
(1) A complex manifold enjoys $\CAP$ if and only if it is Oka.
\\
(2) An algebraic manifold enjoys $\aCAP$ if it is algebraically Oka.
\end{theorem}

\begin{proof}[Proof of Corollary \ref{corollary:complement}]
We may assume that $A\subset Y$ is algebraically degenerate from the beginning.
Let us verify that $Y\setminus A$ enjoys CAP.
Take a compact convex set $K\subset\C^{n}$ and a holomorphic map $f:U\to Y\setminus A$ from an open neighborhood of $K$.
Since $Y$ is algebraically Oka, it enjoys aCAP by Theorem \ref{theorem:CAP}.
Thus there exists a regular map $\tilde f:\C^{n}\to Y$ which approximates $f$ uniformly on $K$ and satisfies $\tilde f(K)\subset Y\setminus A$.
By the jet transversality theorem (Theorem \ref{theorem:main}), we may assume that $\tilde f$ is transverse to $A$.
Then the inverse image $\tilde f^{-1}(A)\subset\C^{n}$ is an algebraically degenerate closed complex subvariety of codimension at least two.
Therefore the complement $\C^{n}\setminus\tilde f^{-1}(A)$ is Oka by the result of Forstneri\v{c} and Prezelj \cite[Theorem 1.6]{Forstneric2002}, and hence it enjoys CAP by Theorem \ref{theorem:CAP}.
Thus there exists a holomorphic map $g:\C^{n}\to\C^{n}\setminus\tilde f^{-1}(A)$ which approximates the inclusion $K\hookrightarrow\C^{n}\setminus\tilde f^{-1}(A)$ uniformly.
Then the composition $\tilde f\circ g:\C^{n}\to Y\setminus A$ approximates $f$ uniformly on $K$.
Therefore $Y\setminus A$ is Oka by Theorem \ref{theorem:CAP}.
\end{proof}

\begin{remark}
\label{remark:aCAP}
Let $Y$ be an algebraic manifold satisfying the assumption in Theorem \ref{theorem:main}.
The above argument also implies that the complement of a closed algebraic subvariety of codimension at least two in $Y$ enjoys aCAP since the complement $\C^{n}\setminus\tilde f^{-1}(A)$ is algebraically Oka in this case (cf. \cite[Proposition 5.6.17]{Forstneric2017}).
It is not known whether aCAP implies the algebraic Oka property (the converse of (2) in Theorem \ref{theorem:CAP}).
\end{remark}

In the previous paper, we proved that the blowup of an algebraically Oka manifold along an algebraic submanifold is Oka \cite[Corollary 1.5]{Kusakabea}.
For an algebraic manifold $Y$ satisfying the assumption in Theorem \ref{theorem:main}, we can reduce the proof of the Oka property of the blowup of $Y$ along an algebraically degenerate complex submanifold to the Euclidean case as in the proof of Corollary \ref{corollary:complement} (see the proof of \cite[Corollary 4.3]{Kusakabea}).
Thus it is natural to ask the following question.
It will be studied in future work.

\begin{question}
Is the blowup of $\C^{n}$ along an algebraically degenerate closed complex submanifold is Oka?
\end{question}

%
%

\section*{Acknowledgement}

I would like to thank Finnur L\'{a}russon for helpful comments.
This work was supported by JSPS KAKENHI Grant Number JP18J20418.

%
%


\end{document}